\documentclass[12pt]{article}

\usepackage{amssymb}
\usepackage{amsmath}
\usepackage{amsthm}
\usepackage{multicol}
\usepackage{url}

\begin{document}

\newtheorem{thm}{Theorem}[section]
\newtheorem{cor}{Corollary}[section]
\newtheorem{lemm}{Lemma}[section]
\newtheorem{conj}{Conjecture}[section]
\theoremstyle{definition}
\newtheorem{defn}{Definition}[section]
\newtheorem{exmpl}{Example}[section]
\newtheorem{remrk}{Remark}[section]

\title{The Abundancy Index of Divisors of\\
Odd Perfect Numbers}

\author{Jose Arnaldo B. Dris \\
Far Eastern University \\
Manila, Philippines \\
jabdris@yahoo.com.ph}

\maketitle

\begin{abstract}
If $N = {q^k}{n^2}$ is an odd perfect number, where $q$ is the Euler prime, then we show that $n < q$ is sufficient for Sorli's conjecture that $k = \nu_{q}(N) = 1$ to hold. We also prove that $q^k < \frac{2}{3}{n^2}$, and that $I(q^k) < I(n)$, where $I(x)$ is the abundancy index of $x$.
\end{abstract}

\section{Introduction}
Perfect numbers are positive integral solutions to the number-theoretic equation $\sigma(N) = 2N$, where $\sigma$ is the sum-of-divisors function.  Euclid derived the general form for the even case;  Euler proved that every even perfect number is given in the Euclidean form $N = {2^{p - 1}}(2^p - 1)$ where $p$ and $2^p - 1$ are prime.  On the other hand, it is still an open question to determine the existence (or otherwise) for an odd perfect number.   Euler proved that every odd perfect number is given in the so-called Eulerian form $N = {q^k}{n^2}$ where $q \equiv k \equiv 1\pmod 4$ and $\gcd(q, n) = 1$.  (We call $q$ the Euler prime of the odd perfect number $N$, and the component $q^k$ will be called the Euler factor of $N$.)  As of February 2012, only 47 even perfect numbers are known (13 of which were found by the distributed computing project GIMPS \cite{Various}), while no single example of an odd perfect number has been found.  (Ochem and Rao of CNRS, France are currently orchestrating an effort to 
push the lower bound for an odd perfect number from the previously known ${10}^{300}$ to a significantly improved ${10}^{1500}$ (see \cite{Ochem})).  Nielsen has obtained the lower bound: $\omega(N) \ge 9$, for the number of distinct prime factors of $N$ (\cite{Nielsen1}); and the upper bound: $N < 2^{4^{\omega(N)}}$ (see \cite{Nielsen2}).

We use the following notations.  Let $\sigma(x)$ denote the sum of the divisors of the natural number $x$.  That is, let $\sigma(x) = \sum_{d|x}{d}$.  Let $\omega(x)$ denote the number of distinct prime factors of $x$.  Let $\nu_{q}(N)$ denote the highest power of $q$ that divides $N$;  that is, if $l = \nu_{q}(N)$, then ${q^l}|N$ but $q^{l+1} \nmid N$.  Let $I(x) = \sigma(x)/x$ denote the abundancy index of $x$.

Sorli conjectured in \cite{Sorli} that the exponent $k = \nu_{q}(N)$ of the Euler prime $q$ for an odd perfect number $N$ given in the Eulerian form $N = {q^k}{n^2}$, is one.

Throughout this paper, we will let $$N = {q^k}{n^2} = \displaystyle\prod_{j=1}^{\omega(N)}{{q_j}^{\beta_j}}$$ denote the canonical factorization of the odd perfect number $N$. That is, 
$$\min(q_j) = {q_1} < {q_2} < {q_3} < \dotsb < {q_{\omega(N)}} = \max(q_j).$$  
Note that $q$ is never the smallest prime divisor of $N$.  This is because $q$, being congruent to $1$ modulo $4$, satisfies $(q + 1) | \sigma(q^k) | \sigma(N) = 2N$ giving $\frac{q+1}{2} | N$, so $N$ must have a smaller odd prime divisor than $q$.

\section{Odd Perfect Numbers Circa 2008}

We begin with the following definition:

\begin{defn}\label{EulerianForm}
An odd perfect number $N$ is said to be given in Eulerian form if $N = {q^k}{n^2}$ where $q \equiv k \equiv 1\pmod 4$ and $\gcd(q, n) = 1$. 
\end{defn}

The author made the following conjecture \cite{Dris2}:

\begin{conj}\label{Conjecture1}
Suppose there is an odd perfect number given in Eulerian form.  Then ${q^k} < n$.
\end{conj}

The author formulated Conjecture \ref{Conjecture1} on the basis of the following result:

\begin{lemm}\label{Lemma1}
If an odd perfect number $N$ is given in Eulerian form, then $I(q^k) < \displaystyle\frac{5}{4} < \displaystyle\sqrt{\frac{8}{5}} < I(n)$.
\end{lemm}

\begin{proof}
Since $q$ is the Euler prime and
$$I(N) = 2 = I(q^k)I(n^2),$$
we appeal to some quick numerical results.  Since
$$I(q^k) < \frac{q}{q - 1}$$
and $q \equiv 1\pmod 4$, we know that $q \ge 5$. Consequently, we have
$$1 < I(q^k) < \frac{5}{4} = 1.25.$$
On the other hand, 
$$I(n^2) = \frac{2}{I(q^k)}$$
so that we obtain the bounds
$$1.6 = \frac{8}{5} < I(n^2) < 2.$$
But it is also well-known (\cite{Nathanson,Sandor1,Sandor2}) that the abundancy index (as a function) satisfies the inequality
$$I(ab) \le I(a)I(b)$$
with equality occurring if and only if $\gcd(a, b) = 1$.

In particular, by setting $a = b = n$, we get
$$\frac{2}{I(q^k)} = I(n^2) < (I(n))^2$$
whereupon we get the lower bound
$$\sqrt{\frac{8}{5}} < \sqrt{\frac{2}{I(q^k)}} = \sqrt{I(n^2)} < I(n).$$
We get the rational approximation $\sqrt{8/5} \approx 1.264911$.
\end{proof}

\begin{remrk}\label{Remark1}
When Conjecture \ref{Conjecture1} was formulated in 2008, the author was under the naive impression that the divisibility constraint $\gcd(q, n) = 1$ induced an ``ordering'' property for the Euler prime-power $q^k$ and the component $n=\sqrt{N/q^k}$, in the sense that the related inequality $q^k < n^2$ followed from the result $I(q^k) < I(n^2)$.  (Indeed, the author was able to derive the (slightly) stronger result $q^k < \sigma(q^k) \le (2/3){n^2}$ \cite{Dris2}).
\end{remrk}

We reproduce the proof for a generalization of the author's result mentioned in Remark \ref{Remark1} in the following theorem.

\begin{thm}\label{Theorem}
Suppose there is an odd perfect number with canonical factorization
$$N = \displaystyle\prod_{i = 1}^{\omega(N)}{{q_i}^{\alpha_i}}$$
where the ${q_i}$'s are primes and $q_1 < q_2 < \ldots < q_{\omega(N)}$.  Then, for all $i$ with $1 \le i \le \omega(N)$, the numbers $\rho_i = \sigma(N/{{q_i}^{\alpha_i}})/{{q_i}^{\alpha_i}}$ are positive integers and satisfy $\rho_i \ge 3$.
\end{thm}

\begin{proof}
Since $$N = \displaystyle\prod_{i = 1}^{\omega(N)}{{q_i}^{\alpha_i}}$$ is an odd perfect number and ${{q_i}^{\alpha_i}} || N$ $\forall i$, then the quantity $\rho_i = \sigma(N/{{q_i}^{\alpha_i}})/{{q_i}^{\alpha_i}}$ is an integer (because $\gcd({q_i}^{\alpha_i}, \sigma({{q_i}^{\alpha_i}})) = 1$).

Suppose $\rho_i = 1$.  Then $\sigma(N/{{q_i}^{\alpha_i}}) = {q_i}^{\alpha_i}$ and $\sigma({q_i}^{\alpha_i}) = {2N}/{{q_i}^{\alpha_i}}$.  Since $N$ is an odd perfect number, $q_i$ is odd, whereupon we have an odd $\alpha_i$ by considering parity conditions from the last equation.  But this means that $q_i$ is the Euler prime $q$, and we rewrite the equations using ${q_i}^{\alpha_i} = q^k$ and $N/{{q_i}^{\alpha_i}} = N/{q^k} = n^2$, giving $\sigma(q^k) = 2{n^2}$ and $\sigma(n^2) = q^k$.  This contradicts Dandapat, et.~al. \cite{Dandapat} who showed in 1975 that no odd perfect number satisfies these constraints.  This implies that $\rho_i \ne 1$.

Suppose $\rho_i = 2$.  Then $\sigma(N/{{q_i}^{\alpha_i}}) = 2{q_i}^{\alpha_i}$ and $\sigma({q_i}^{\alpha_i}) = N/{{q_i}^{\alpha_i}}$.  Since $N/{{q_i}^{\alpha_i}}$ is odd, then the last equation gives $\alpha_i$ is even.  Applying the $\sigma$ function to both sides of the last equation, we get $\sigma(\sigma({q_i}^{\alpha_i})) = \sigma(N/{{q_i}^{\alpha_i}}) = 2{q_i}^{\alpha_i}$.  This last equation implies that ${q_i}^{\alpha_i}$ is superperfect.  This contradicts Suryanarayana \cite{Suryanarayana} who showed in 1973 that ``There is no odd superperfect number of the form $p^{2\alpha}$'' (where $p$ is prime).  This implies that $\rho_i \ne 2$.  Since $\rho_i \in \mathbb{N}$, $\rho_i \ge 3$ and we are done.
\end{proof}

\begin{cor}\label{Corollary}
If an odd perfect number $N$ is given in Eulerian form, then $q^k < (2/3)n^2$.
\end{cor}

Next, we define the functions $L(q)$ and $U(q)$.

\begin{defn}\label{LowerandUpperBounds}
If $q$ is the Euler prime of an odd perfect number $N$ given in Eulerian form, then 
$$L(q) = (3q^2 -4q + 2)/(q(q - 1))$$ 
and 
$$U(q) = (3q^2 + 2q + 1)/(q(q + 1)).$$  
\end{defn}

The author obtained the following results in the same year (2008).

\begin{lemm}\label{Lemma2}
Let $N$ be an odd perfect number given in Eulerian form.  Then we have the bounds $L(q) < I(q^k) + I(n^2) \le U(q)$.
\end{lemm}

\begin{proof}
Starting from the (trivial) inequalities
$$\frac{q + 1}{q} \le I(q^k) < \frac{q}{q - 1}$$
we get
$$\frac{2(q - 1)}{q} < I(n^2) = \frac{2}{I(q^k)} \le \frac{2q}{q + 1}.$$
Notice that
$$\frac{q}{q - 1} < \frac{2(q - 1)}{q}$$
for $q$ an Euler prime. Consequently
$$I(q^k) < I(n^2)$$
a result which was mentioned earlier in Remark \ref{Remark1}.

Consider the product $\left(I(q^k) - \displaystyle\frac{q + 1}{q}\right)\left(I(n^2) - \displaystyle\frac{q + 1}{q}\right)$. This product is nonnegative since $\displaystyle\frac{q + 1}{q} \le I(q^k) < I(n^2)$.  Expanding the product and simplifying using the equation $I(q^k)I(n^2) = 2$, we get the upper bound $U(q) = \displaystyle\frac{3q^2 + 2q + 1}{q(q + 1)}$ for the sum $I(q^k) + I(n^2)$.

Next, consider the product $\left(I(q^k) - \displaystyle\frac{q}{q - 1}\right)\left(I(n^2) - \displaystyle\frac{q}{q - 1}\right)$. This product is negative since $I(q^k) < \displaystyle\frac{q}{q - 1} < I(n^2)$. Again, expanding the product and simplifying using the equation $I(q^k)I(n^2) = 2$, we get the lower bound $L(q) = \displaystyle\frac{3q^2 -4q + 2}{q(q - 1)}$ for the same sum $I(q^k) + I(n^2)$.

A quick double-check gives you that, indeed, the lower bound $L(q)$ is less than the upper bound $U(q)$, if $q$ is an Euler prime.
\end{proof}

\begin{remrk}\label{Remark2}
Notice that, from the proof of Lemma \ref{Lemma2}, we have
$$\frac{q}{q - 1} < \frac{2(q - 1)}{q}$$
which implies that ${\left({\displaystyle\frac{q}{q - 1}}\right)}^2 < 2$. Thus
$$1 < I(q^k) < \frac{q}{q - 1} < \sqrt{2} = \frac{2}{\sqrt{2}} < \frac{2(q - 1)}{q} < I(n^2) < 2.$$
Also, observe from Lemma \ref{Lemma1} that
$$I(q^k) < \frac{5}{4} < \sqrt{\frac{8}{5}} < \sqrt{\frac{2}{I(q^k)}}$$
which implies that $I(q^k)\sqrt{I(q^k)} < \sqrt{2}.$  It follows that
$$I(q^k) < \sqrt[3]{2}.$$
We get the rational approximation $\sqrt[3]{2} \approx 1.259921$.
\end{remrk}

We give explicit bounds for the sum $I(q^k) + I(n^2)$ in the following corollary.

\begin{cor}\label{Corollary1}
Let $N$ be an odd perfect number given in Eulerian form. Then we have the following (explicit) numerical bounds:
$$2.85 = \frac{57}{20} < I(q^k) + I(n^2) < 3$$
with the further result that they are best-possible.
\end{cor}

\begin{proof}
This corollary can be proved using Lemma \ref{Lemma2} and basic differential calculus, and is left as an exercise to the interested reader.
\end{proof}

\begin{remrk}\label{Remark3}
As remarked by Joshua Zelinsky in 2005: ``Any improvement on the upper bound of 3 would have (similar) implications for all arbitrarily large primes and thus would be a very major result.'' (e.g. $L(q) < 2.99$ implies $q \le 97$.)  In this sense, the inequality $$2.85 = \frac{57}{20} < I(q^k) + I(n^2) < 3$$ is best-possible.
\end{remrk}

\begin{remrk}\label{Remark4}
Note that, from Lemma \ref{Lemma2}, 
$$L(q) = \frac{3q^2 -4q + 2}{q(q - 1)} = 3 - \frac{q - 2}{q(q - 1)}$$ 
and 
$$U(q) = \frac{3q^2 + 2q + 1}{q(q + 1)} = 3 - \frac{q - 1}{q(q + 1)}.$$
Observe that, when $L(x)$ and $U(x)$ are viewed as functions on the domain $D = \mathbb{R}\setminus\{-1, 0, 1\}$, then 
$$L(x + 1) = U(x)$$ 
and 
$$U(2) = U(3) = L(3) = \frac{17}{6} < 2.84.$$
\end{remrk}

\section{Sorli's Conjecture [2003]}

We now state Sorli's conjecture on odd perfect numbers:

\begin{conj}\label{Conjecture2}
If $N$ is an odd perfect number with Euler prime $q$ then $q || N$.
\end{conj}

\begin{remrk}\label{Remark6}
In other words, if the odd perfect number $N$ is given in the Eulerian form $N = {q^k}{n^2}$, then Sorli's conjecture predicts that $k = \nu_{q}(N) = 1$.  Note that, in general by Remark \ref{Remark1} we have 
$$q^k<\sqrt{N}={q^{k/2}}n$$
which gives $q^{k/2} < n$. Sorli's conjecture, if proved, will enable easier computations with odd perfect numbers because then the abundancy index $I(q^k)$ for the Euler factor $q^k$ collapses to $I(q) = (q + 1)/q$.
\end{remrk}

We give a set of sufficient conditions for Sorli's conjecture to hold. (In that direction, recall that the components $q^k$ and $n^2$ of the odd perfect number $N = {q^k}{n^2}$ are related via the inequality $q^k < n^2$, as mentioned in Remark \ref{Remark1}.)

\begin{lemm}\label{Lemma3}
Let $N$ be an odd perfect number given in Eulerian form.  If $n < q$, then $k = 1$.
\end{lemm}

\begin{proof}
If $n < q$, then by Corollary \ref{Corollary}, $q \le q^k < n^2 < q^2$ so $k = 1$. 
\end{proof}

\begin{remrk}\label{Remark7} Via a similar argument, we get that $n < q^2$ also implies $k = 1$.
\end{remrk}

\begin{lemm}\label{Lemma4}
Let $N$ be an odd perfect number given in Eulerian form.  If $\sigma(n) \leq \sigma(q)$, then $k = 1$.
\end{lemm}

\begin{proof}
Suppose that $\sigma(n) \leq \sigma(q)$.  Since $I(q) < I(n)$ by Lemma \ref{Lemma1}, it follows that $\displaystyle\frac{\sigma(q)}{\sigma(n)} < \displaystyle\frac{q}{n}$.  By our assumption, $1 \leq \displaystyle\frac{\sigma(q)}{\sigma(n)}$, whereupon we get $n < q$.  Therefore, Lemma \ref{Lemma3} gives $k = 1$. 
\end{proof}

\begin{lemm}\label{Lemma5}
Let $N$ be an odd perfect number given in Eulerian form.  If $\displaystyle\frac{\sigma(n)}{q} < \displaystyle\frac{\sigma(q)}{n}$, then $k = 1$.
\end{lemm}

\begin{proof}
If $\displaystyle\frac{\sigma(n)}{q} < \displaystyle\frac{\sigma(q)}{n}$, then since $\displaystyle\frac{\sigma(q)}{q} < \displaystyle\frac{\sigma(n)}{n}$, it follows that $\displaystyle\frac{\sigma(n) + \sigma(q)}{q} < \displaystyle\frac{\sigma(q) + \sigma(n)}{n}$, whereupon we get $n < q$.  By Lemma \ref{Lemma3}, we have $k = 1$.
\end{proof}

\begin{thm}\label{Theorem2}
Let $N$ be an odd perfect number given in Eulerian form.  Then $n < q$ if and only if $N < q^3$.
\end{thm}

\begin{proof}
Suppose that ${q^k}{n^2} = N < q^3$.  Then $n^2 < q^{3 - k} \leq q^2$, which implies that $n < q$.  We prove the other direction via the contrapositive.  Suppose that $q < N^{1/3}$.  We want to show that $q < n$.  Assume to the contrary that $n < q$.  By Lemma \ref{Lemma3}, $k = 1$.  Therefore, we have $$q < N^{1/3} = {q^{k/3}}{n^{2/3}} = {q^{1/3}}{n^{2/3}} < {q^{1/3}}{q^{2/3}} = q$$ which is a contradiction.
\end{proof}

\begin{remrk}\label{Remark8}
If $N$ is an odd perfect number given in Eulerian form, then since $q^k < n^2$ by Corollary \ref{Corollary}, we have $q^2 \leq q^{2k} < N$.  Recently in \cite{AcquaahKonyagin}, Acquaah and Konyagin have been able to show that the Euler prime $q$ satisfies $q < (3N)^{1/3}$.
\end{remrk}

\section{Acknowledgment}
The author sincerely thanks the anonymous referees who have made several corrections and suggestions, which helped in improving the style of the paper.  The author would also like to thank Severino Gervacio, Arlene Pascasio, Fidel Nemenzo, Jose Balmaceda, Julius Basilla, Blessilda Raposa, Carl Pomerance, Richard Brent and Gerry Myerson for sharing their expertise.  Lastly, the author wishes to express his deepest gratitude to Joshua Zelinsky from whom he derived the inspiration to pursue odd perfect numbers as a research topic.

\bigskip
\hrule
\bigskip

\noindent 2010 {\it Mathematics Subject Classification}:
Primary 11A05; Secondary 11J25, 11J99.

\noindent \emph{Keywords: } 
odd perfect number, Sorli's conjecture, Euler prime.

\bigskip
\hrule
\bigskip

\end{document}